\documentclass{amsart}

\usepackage{amsmath,amssymb,amsthm}

\usepackage{pinlabel}

\newtheorem{prop}{Proposition}[section]
\newtheorem{thm}[prop]{Theorem}
\newtheorem{lem}[prop]{Lemma}
\newtheorem{cor}[prop]{Corollary}

\theoremstyle{definition}

\newtheorem{defn}[prop]{Definition}
\newtheorem{ex}[prop]{Example}

\newtheorem{rem}[prop]{Remark}

\newtheorem*{ack}{Acknowledgements}

\def\co{\colon\thinspace}

\newcommand{\alphast}{\alpha_{\mathrm{st}}}

\newcommand{\C}{\mathbb{C}}
\newcommand{\Ca}{C_{\mathrm{aff}}}
\newcommand{\CP}{\mathbb{C}\mathrm{P}}

\newcommand{\rmd}{\mathrm{d}}

\newcommand{\rme}{\mathrm{e}}

\newcommand{\rmi}{\mathrm{i}}

\newcommand{\N}{\mathbb{N}}

\newcommand{\omegast}{\omega_{\mathrm{st}}}

\newcommand{\R}{\mathbb{R}}
\newcommand{\Rst}{R_{\mathrm{st}}}

\newcommand{\oSigma}{\overline{\Sigma}}

\newcommand{\ox}{\overline{x}}

\newcommand{\oy}{\overline{y}}

\newcommand{\oz}{\overline{z}}
\newcommand{\Z}{\mathbb{Z}}

\DeclareMathOperator{\Int}{Int}

\begin{document}

\author[P.~Albers]{Peter Albers}
\address{Mathematisches Institut, Universit\"at Heidelberg,
Im Neuenheimer Feld 205, 69120 Heidelberg, Germany}
\email{palbers@mathi.uni-heidelberg.de}

\author[H.~Geiges]{Hansj\"org Geiges}
\address{Mathematisches Institut, Universit\"at zu K\"oln,
Weyertal 86--90, 50931 K\"oln, Germany}
\email{geiges@math.uni-koeln.de}

\author[K.~Zehmisch]{Kai Zehmisch}
\address{Mathematisches Institut, Justus-Liebig-Universit\"at Gie{\ss}en,
Arndtstra{\ss}e 2, 35392 Gie{\ss}en, Germany}
\email{kai.zehmisch@math.uni-giessen.de}

\title[Symplectic dynamics and the degree--genus formula]{A symplectic
dynamics proof of the\\degree--genus formula}

\thanks{This research is part of a project
in the SFB/TRR 191 `Symplectic Structures in Geometry, Algebra and Dynamics',
funded by the DFG}

\date{}

\begin{abstract}
We classify global surfaces of section for the Reeb flow of
the standard contact form on the
$3$-sphere, defining the Hopf fibration. As an application,
we prove the degree-genus formula for complex projective curves, using
an elementary degeneration process inspired by the language of
holomorphic buildings in symplectic field theory.
\end{abstract}



\maketitle


\section{Introduction}
A \emph{global surface of section} for the flow of
a non-singular vector field $X$ on a three-manifold $M$
is an embedded compact surface $\Sigma\subset M$ such that
\begin{enumerate}
\item[(i)] the boundary $\partial\Sigma$ is a union of orbits;
\item[(ii)] the interior $\Int(\Sigma)$ is transverse to~$X$;
\item[(iii)] the orbit of $X$ through any point in $M\setminus\partial\Sigma$
intersects $\Int(\Sigma)$ in forward and backward time.
\end{enumerate}

If one can find such a global surface of section, understanding
the dynamics of $X$ essentially reduces to studying the Poincar\'e
return map $\Int(\Sigma)\rightarrow\Int(\Sigma)$,
which sends each point $p\in\Int(\Sigma)$ to the first intersection point
of the $X$-orbit through $p$ with $\Int(\Sigma)$ in forward time.

In symplectic dynamics, where $X$ is a Hamiltonian or Reeb vector field,
there are a number of results on the existence or non-existence of
global surfaces of section, e.g.\
\cite{hwz98,hwz03,hms15,hrsa11,hrsa18,koer19}. Conversely, one can
ask for the existence of flows with a given surface of section
and return map. For instance, in \cite{agz18} we describe
a construction of Reeb flows on the $3$-sphere $S^3$ with a disc-like global
surface of section, where the return map is a pseudorotation; see
also~\cite{abhs18,hutc16}.

For Reeb flows on the $3$-sphere coming from contact forms that define
the standard tight contact structure, the following are the main facts known
about the existence of global surfaces of section. Hofer, Wysocki and Zehnder
\cite[Theorem~1.3]{hwz98} give a sufficient criterion (dynamical convexity)
for the existence of a disc-like global surface of section.
Hryniewicz and Salom\~ao~\cite[Theorem~1.3]{hrsa11} describe a necessary
and sufficient condition for a periodic Reeb orbit of a non-degenerate
contact form to bound a disc-like global surface of section. A Reeb flow
without a disc-like global surface of section has been constructed
by van Koert~\cite{koer19}. It is not known if there is
a Reeb flow (in the described class) without any global surface of section.

This motivates the question whether one can give a complete classification
of global surfaces of section for a given flow.
In the present paper, we consider the Hopf flow on the $3$-sphere
$S^3\subset\C^2$, that is, the flow
\begin{equation}
\label{eqn:hopf-flow}
\Psi_t\co (z_1,z_2)\longmapsto(\rme^{\rmi t}z_1,\rme^{\rmi t}z_2),
\;\;\; t\in\R,
\end{equation}
defining the Hopf fibration $S^3\rightarrow S^2$,
as well as the induced flows on the
lens space quotients $L(d,1)$ of~$S^3$. Our first main result, which is purely
topological, gives a classification, up to isotopy, of the surfaces that can
arise as global surfaces of section for these flows.

The second main result is a symplectic dynamics proof of the
classical degree-genus formula for complex projective curves.
This formula says that a non-singular complex algebraic curve of degree $d$
in the projective plane $\CP^2$ is topologically a connected,
closed, oriented surface of genus
\begin{equation}
\label{eqn:dg}
g=\frac{1}{2}(d-1)(d-2).
\end{equation}
Our proof uses degenerations of complex projective curves in the
spirit of Symplectic Field Theory (SFT).
Perhaps surprisingly, the SFT point of view elucidates why
(\ref{eqn:dg}) should be read as a sum $\sum_{k=1}^{d-2}k$.
For a given non-singular complex projective curve
of degree~$d$, we describe a $1$-dimensional family of curves
starting at the given one and converging to a holomorphic building
of height~$d$ in the sense of~\cite{behwz03}.
Each level in this holomorphic building has genus~$0$, and the gluing
of level $k+2$ to level $k+1$ contributes $k$ to the genus, $k=1,\ldots,
d-2$ (see Figure~\ref{figure:building}).
This may be regarded as a motivating
example for the degenerations studied in SFT.
We ought to point out that we do not use any actual results from SFT.

The `standard' proof of the degree-genus formula,
using branched coverings and the Riemann--Hurwitz formula,
can be found in~\cite[Chapter~4]{kirw92}; see also \cite[\S21]{prso97}
and~\cite[p.~219]{grha78}.

Here is an outline of the paper.
In Section~\ref{section:hopf-flow} we construct some examples
of surfaces of section for the Hopf flow. We show how certain equivalences
between Seifert invariants can be interpreted as modifications
of such surfaces.

In Section~\ref{section:lens} we relate global surfaces of section
for the Hopf flow on $S^3$ to those for the induced flow (which likewise
defines an $S^1$-fibration)
on the lens space quotients $L(d,1)$. We then classify $1$-sections
in $L(d,1)$, i.e.\ global surfaces of section that intersect each
fibre exactly once. The classification of
$d$-sections (Definition~\ref{defn:d-section})
for the Hopf flow on $S^3$ with all boundary
orbits traversed positively is achieved in Section~\ref{section:pos-d}.

In Section~\ref{section:curves} we discuss a number
of examples how algebraic curves in
$\CP^2$ give rise to global surfaces of section for the Hopf flow.
This allows one to determine the genus of these particular curves.

Finally, in Section~\ref{section:buildings} we prove the degree-genus
formula, using genericity properties of algebraic curves. We give one
proof directly from the classification of global surfaces of section
for the Hopf flow. The second, more instructive proof, uses
SFT type degenerations to interpret the degree-genus formula as
as a sum $\sum_{k=1}^{d-2} k$. Technical details of the SFT type
convergence are relegated to Section~\ref{section:convergence}.
\section{The Hopf flow}
\label{section:hopf-flow}
Our aim is to describe surfaces of section for the Hopf flow
(\ref{eqn:hopf-flow}) on the $3$-sphere~$S^3$. Thinking of $S^3$ as
the unit sphere in~$\R^4$, we can define a $1$-form $\alphast$
on $S^3$ by
\begin{equation}
\label{eqn:alphast}
\alphast=\bigl(x_1\,\rmd y_1-y_1\,\rmd x_1+
x_2\,\rmd y_2-y_2\,\rmd x_2\bigr)|_{TS^3}.
\end{equation}
This $1$-form is a contact form, in the sense that
$\alphast\wedge\rmd\alphast$ is a volume form;
$\alphast$ is called the standard contact form on~$S^3$.
The Reeb vector field $\Rst$
of this contact form is defined by $i_{\Rst}\rmd\alphast=0$ and
$\alphast(\Rst)=1$.
Here this means that
\[ \Rst=x_1\partial_{y_1}-y_1\partial_{x_1}+
x_2\partial_{y_2}-y_2\partial_{x_2},\]
which is the vector field giving rise to the Hopf flow.

From this interpretation of the Hopf flow as a Reeb flow, and the
contact condition $\alphast\wedge\rmd\alphast\neq 0$, we see that
the $2$-form $\rmd\alphast$ defines an exact area
form transverse to the flow of~$\Rst$, so any surface of section must have
non-empty boundary.

For more on the basic notions of contact geometry see~\cite{geig08}.
\subsection{$d$-Sections}
Let $\Sigma\subset S^3$ be a surface of section for the Hopf flow.
Then $\partial\Sigma$ is a collection of fibres of the
Hopf fibration $S^3\rightarrow S^2$ over a finite number
of points $p_1,\ldots,p_k\in S^2$. The interior of $\Sigma$
projects surjectively to the connected set $S^2\setminus\{p_1,\ldots, p_k\}$.
It follows that each fibre over this set intersects $\Int(\Sigma)$
in the same number of points, and $\Sigma$
is a $d$-section for some $d\in\N$, in the following sense.

\begin{defn}
\label{defn:d-section}
We call an embedded surface $\Sigma\subset S^3$ a
\emph{$d$-section} for the flow of $\Rst$ if every simple orbit of~$\Rst$
(i.e.\ every Hopf fibre) intersects $\Int(\Sigma)$ in exactly $d$ points or
is a component of~$\partial\Sigma$; the latter will be referred to
as \emph{boundary fibres}. We shall always orient $\Sigma$
such that the $\Rst$-flow intersects $\Sigma$ positively.
The $d$-section is said to be \emph{positive} if the
boundary orientation of $\partial\Sigma$ coincides with the $\Rst$-direction.
\end{defn}

In some examples we shall construct such $d$-sections by starting from
an honest multi-section of the Hopf fibration over $S^2$ with
a certain number of discs removed, and then extending it
to become tangent to the fibres over the centres of these discs,
by gluing in helicoidal surfaces.
\subsection{Examples of $d$-sections}
\label{subsection:ex-d}
We think of $S^3$ as being obtained by gluing two copies $V_1,V_2$
of the solid torus $S^1\times D^2$. Write $\mu_i$
for the meridian and $\lambda_i=S^1\times\{*\}$, with $*\in\partial D^2$,
for the standard longitude on $\partial V_i$. We shall use those same
symbols for any curve on $\partial V_i$ in the same isotopy class.
The gluing described by $\mu_1=\lambda_2$, $\lambda_1=\mu_2$
yields~$S^3$.

More intrinsically, if one thinks of $S^3$ as the unit sphere
in~$\C^2$, we can define $V_i$ as the solid torus given by
$\{|z_i|\leq\sqrt{2}/2\}$. The identification of $V_1$ with $S^1\times D^2$
is given by
\[ V_1=\Bigl\{\bigl(z,\sqrt{1-|z|^2}\,\rme^{\rmi\theta}\bigr)\co
|z|\leq\sqrt{2}/2,\, \theta\in\R/2\pi\Z\Bigr\}.\]
The soul of $V_1$ is
\[ C_1=\bigl\{ (0,\rme^{\rmi\theta})\co \theta\in\R/2\pi\Z\bigr\},\]
corresponding to $S^1\times\{0\}\subset S^1\times D^2$. The solid torus $V_2$
and its soul $C_2$ are defined analogously. The $\mu_i$ and $\lambda_i$ are
\[ \mu_1=\biggl\{\Bigl(\frac{\sqrt{2}}{2}\rme^{\rmi\theta},
\frac{\sqrt{2}}{2}\Bigr)\co\theta\in\R/2\pi\Z\biggr\}=\lambda_2\]
and
\[ \lambda_1=\biggl\{\Bigl(\frac{\sqrt{2}}{2},
\frac{\sqrt{2}}{2}\rme^{\rmi\theta}\Bigr)\co\theta\in\R/2\pi\Z\biggr\}=\mu_2.\]

The two souls $C_1,C_2$ form a positive Hopf link, i.e.\
the two unknots have linking number $+1$. The \emph{Hopf tori}
\[ T^2_r=\bigl\{(z_1,z_2)\in S^3\co |z_1|=r\bigr\},\;\;
r\in ]0,1[,\]
foliate the complement of $C_1,C_2$ in~$S^3$.

In these coordinates, the Hopf flow is simply the flow of $\partial_{\varphi_1}
+\partial_{\varphi_2}$, where $\varphi_i$ is the angular coordinate in the
$z_i$-plane.
The Hopf fibration is then made up of the souls $C_i=S^1\times\{0\}$
of the two solid tori and the $(1,1)$-curves on the Hopf tori,
i.e.\ curves in the class $h:=\mu_1+\lambda_1=\mu_2+\lambda_2$.
\subsubsection{A disc-like $1$-section}
\label{subsubsection:disc}
The disc
\[ \Bigl\{\bigl(\sqrt{1-r^2},r\rme^{\rmi\theta}\bigr)\co r\in[0,1],\;
\theta\in\R/2\pi\Z\Bigr\}
\subset S^3\]
bounded by $C_1$ is a positive $1$-section for the Hopf flow.

Alternatively, we may identify $V_1,V_2$ with solid tori
such that the Hopf fibres correspond to the $S^1$-fibres in $S^1\times D^2$,
so that the fibre class is now represented by $h=S^1\times\{*\}$;
this change in identification amounts to a Dehn twist of the solid
torus along a meridional disc.

The meridional disc in $V_2$ defines a $1$-section for the Hopf flow in
that solid torus. The boundary $\mu_2$ of this disc
is identified with $\lambda_1=h-\mu_1$ in $\partial V_1$. In
$V_1$ we have a helicoidal surface $A$ with oriented
boundary $C_1\sqcup -(h-\mu_1)$, see Figure~\ref{figure:helicoid}.
This annulus $A$ glues with the meridional disc in $V_2$ to form a
positive $1$-section for the Hopf flow.

\begin{figure}[h]
\labellist
\small\hair 2pt
\pinlabel $C_1$ [l] at 115 411
\pinlabel $-h+\mu_1$ [l] at 222 234
\endlabellist
\centering
\includegraphics[scale=.4]{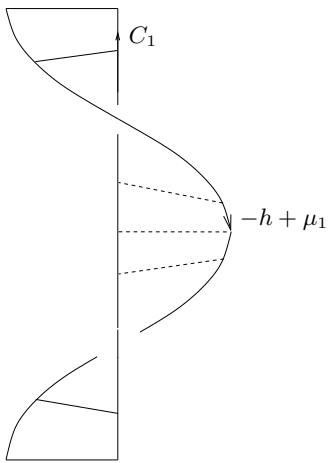}
  \caption{A helicoidal annulus.}
  \label{figure:helicoid}
\end{figure}
\subsubsection{An annular $2$-section}
\label{subsubsection:annular2}
In $V_1$ we find a helicoidal annulus $A_1$ with boundary
$\partial A_1=C_1\sqcup -(h-2\mu_1)$, with $\Int(A_1)$ intersecting each
Hopf fibre positively in two points. Likewise, we have such an annulus
$A_2$ in $V_2$ with $\partial A_2=C_2-(h-2\mu_2)$. Since
\[ h-2\mu_1=\lambda_1-\mu_1=-(\lambda_2-\mu_2)=-(h-2\mu_2),\]
$A_1$ and $A_2$ glue to form a positive annular $2$-section for the Hopf flow.
\subsection{The Euler number}
For the existence of positive $d$-sections, the sign of the
Euler number of the Hopf fibration is crucial.

\begin{lem}
The Hopf fibration, regarded as an $S^1$-bundle over~$S^2$, has
Euler number $e=-1$.
\end{lem}

\begin{proof}
We think of $V_1$ as in Section~\ref{subsubsection:disc} as
a solid torus $S^1\times D^2$ with the Hopf fibres given by
$S^1\times\{*\}$.
The helicoidal surface $A\subset V_1$ described in
Section~\ref{subsubsection:disc} and Figure~\ref{figure:helicoid},
together with the meridional 
disc in~$V_2$, can be turned into a section of the disc bundle
associated with the Hopf bundle by scaling in the fibre
direction,
\[ V_1=S^1\times D^2\supset\Int(A)\ni\bigl(a(p),p\bigr)\longmapsto
\bigl(|p|\cdot a(p),p\bigl)\in D^2\times D^2,\]
and extending to a section with a single zero at $p=0$.
Since $\partial A\cap \partial V_1=\mu_1-h$ makes one
negative twist in the fibre direction as we go once
along the boundary $\mu_1$ of the base disc (i.e. the
second $D^2$-factor), this rescaled section,
seen as a vector field on~$D^2$, has an index $-1$ singularity,
which means that it cuts the zero section in a single
negative point.
\end{proof}
\subsection{The Hopf fibration as a Seifert fibration}
\label{subsection:hopf-seifert}
In Sections \ref{subsection:pp1} and \ref{subsection:class1}
we are going to show how different descriptions of the Hopf fibration
as a Seifert fibration give rise to global surfaces of section
with different numbers of boundary components. Here we give a bare
bones introduction to Seifert invariants. All necessary background
on Seifert fibrations can be found in \cite[Section~2]{gela18};
for a comprehensive treatment see~\cite{jane83}.

Consider again $S^3$ as being obtained by gluing two solid tori
$V_1,V_2$. In terms of meridians $\mu_1,\mu_2$ and longitudes $h$
(on both solid tori), the identification of $\partial V_1$ with
$\partial V_2$ is given by $\mu_1=\lambda_2=-\mu_2+h$ and $h=h$.
The curve $-\mu_2$ is the \emph{negative} boundary of the section in $V_2$
given by a meridional disc. Following the standard conventions for
Seifert invariants, see~\cite{gela18}, the gluing of
the neighbourhoods of the distinguished fibres in a Seifert manifold
should indeed be described with respect to the negative boundary
of the section away from the distinguished fibres (which, in the
general Seifert setting, include all multiple fibres). This means that
the described gluing corresponds to writing $S^3$ as the Seifert manifold
$S^3=M(0;(1,1))$. Here $0$ is the genus of the base $S^2$, and $(1,1)$
are the coefficients of $-\mu_2$ and $h$ in the expression for~$\mu_1$.

For a Seifert manifold
$M\bigl(g;(\alpha_1,\beta_1),\ldots,(\alpha_k,\beta_k)\bigr)$,
the Euler number is defined as $-\sum_i\alpha_i/\beta_i$,
see~\cite{jane83}. This is consistent with our calculation of
the Euler number of the Hopf fibration.

Given such a Seifert manifold
$M\bigl(g;(\alpha_1,\beta_1),\ldots,(\alpha_k,\beta_k)\bigr)$,
one can obtain equivalent descriptions by adding or deleting any pair
$(\alpha,\beta)=(1,0)$, or by replacing each $(\alpha_i,\beta_i)$ by
$(\alpha_i,\beta_i+n_i\alpha_i)$, where $\sum_i n_i=0$.
For instance, the Hopf fibration can alternatively be described as
\begin{equation}
\label{eqn:hopf3}
M\bigl(0;(1,1),(1,1), (1,-1)\bigr),
\end{equation}
by first adding two pairs $(1,0)$, and then replacing them
by $(1,1)$ and $(1,-1)$.
\subsection{A pair of pants $1$-section}
\label{subsection:pp1}
We now want to show that the description (\ref{eqn:hopf3}) of the
Hopf fibration as a Seifert fibration with three distinguished
fibres (albeit of multiplicity~$1$) gives rise to
a pair of pants $1$-section with one negative and two positive
boundary components, i.e.\ one component where the boundary orientation
is the opposite of the direction of the Hopf flow, and two components
where the orientations coincide. An alternative construction illustrates
the equivalences between Seifert invariants in terms of a
modification of the surface of section. We also describe a third construction
that we shall take up again in Section~\ref{subsubsection:pos3}.

\vspace{1mm}

(i) The description (\ref{eqn:hopf3}) means that we start with
a $2$-sphere with three open discs removed, i.e.\ a pair of pants $P$.
Over $P$ the Hopf bundle is the trivial bundle $S^1\times P$,
and we take a constant section there (which we identify with~$P$).
Write the negatively oriented boundary $-\partial P$ of $P$ as
\[ -\partial P=\sigma_1\sqcup\sigma_2\sqcup\sigma_3.\]
We now glue three solid tori $V_1,V_2,V_3$ to $S^1\times P$ with gluing map
\[ \mu_i=\sigma_i+\beta_ih,\;\; h=h,\]
where $\beta_1=\beta_2=1$ and $\beta_3=-1$. In $V_i$ we find
a helicoidal annulus $A_i$ with oriented boundary $\partial A_i=
\beta_i C_i\sqcup\sigma_i$, where $C_i$ is the soul of~$V_i$.
These three annuli can be glued to $P$ along the $\sigma_i$
to yield the desired $1$-section.

\vspace{1mm}

(ii) Alternatively, we can start with a disc-like positive $1$-section
$\Sigma$ for the Hopf flow and modify it as follows. Choose a disc
$D^2_0\subset\Int(\Sigma)$. The Hopf fibres passing through $D^2_0$
define a trivial bundle $S^1\times D^2_0\rightarrow D^2_0$. Remove the
interior of two disjoint discs $D^2_2$ and $D^2_3$ from the interior
of~$D^2_0$, leaving us with a product bundle over a pair of pants~$P$.
In $S^1\times P$ we find a vertical annulus $A$ with oriented boundary
equal to a positive fibre in $S^1\times\partial D^2_2$
and a negative fibre in $S^1\times\partial D^2_3$.
This annulus can be assumed
to intersect the constant section $P$ in a simple curve $\gamma$ joining
$\partial D^2_2$ with $\partial D^2_3$.

\begin{figure}[h]
\labellist
\small\hair 2pt
\pinlabel $A$ [l] at 72 133
\pinlabel $P$ [b] at 131 75
\endlabellist
\centering
\includegraphics[scale=.5]{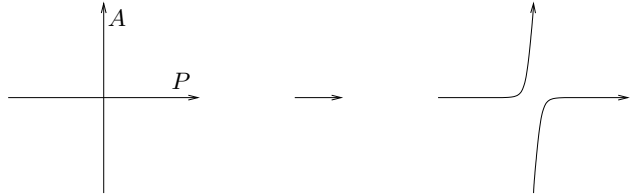}
  \caption{Gluing $A$ and $P$.}
  \label{figure:glue}
\end{figure}

By slicing open both $P$ and $A$ along $\gamma$ and regluing them as
illustrated (in a cross section)
in Figure~\ref{figure:glue}, one obtains a $1$-section
with helicoidal boundary curves
\[  -\sigma_2=-\mu_2 +h\;\;\text{and}\;\;-\sigma_3=-\mu_3-h\]
on $S^1\times\partial D^2_2$ and $S^1\times\partial D^2_3$,
respectively. This $1$-section projects diffeomorphically onto~$P$
(away from the boundary curves), so it is still a pair of pants.

By gluing in helicoidal annuli $A_i$ in $S^1\times D^2_i$
with boundary $\partial A_2=C_2\sqcup \sigma_2$ and
$\partial A_3=-C_3\sqcup \sigma_3$, where $C_i=S^1\times\{0\}$
is the central fibre of $S^1\times D^2_i$, we obtain again the desired
$1$-section.
This second construction actually explains
the equivalences of Seifert invariants that led from
the description of the Hopf fibration as $M\bigl(0;(1,1)\bigr)$
to that in~(\ref{eqn:hopf3}).

\vspace{1mm}

(iii) Here is a third method of construction, which will be useful later on.
This time we think of $S^3$, as at the beginning of
Section~\ref{subsection:ex-d}, as a gluing of two solid tori
$V_1,V_2$ with the identification $\mu_1=\lambda_2$, $\lambda_1=\mu_2$.
The Hopf fibration is given by the two souls $C_1,C_2$ and the
$(1,1)$-curves on the Hopf tori parallel to $\partial V_1=\partial V_2$.

\begin{figure}[h]
\labellist
\small\hair 2pt
\pinlabel $\lambda_1$ [r] at 0 105
\pinlabel $\mu_1$ [t] at 105 0
\pinlabel $\lambda_2$ [r] at 323 105
\pinlabel $\mu_2$ [t] at 428 0
\pinlabel $a_1$ [bl] at 51 170
\pinlabel $h_1$ [tl] at 64 64
\pinlabel $a_2$ [tr] at 484 64
\pinlabel $h_2$ [br] at 493 170
\endlabellist
\centering
\includegraphics[scale=.5]{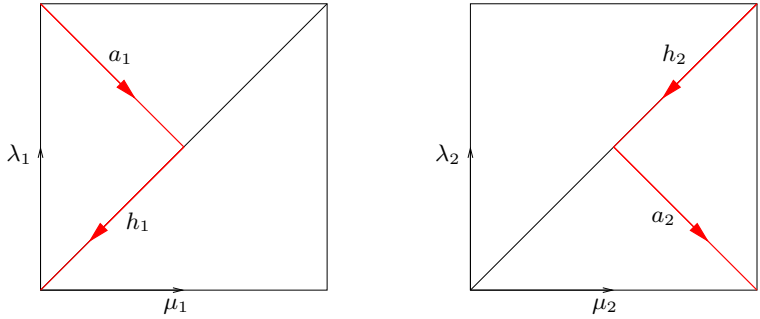}
  \caption{A pair of pants $1$-section.}
  \label{figure:pants}
\end{figure}

The simple closed curve $a_1+h_1$ on $\partial V_1$, shown in
Figure~\ref{figure:pants}, is homotopic to $-\lambda_1$. This allows us
to find an annulus $A_1$ in $V_1$ with boundary $\partial A_1=C_1
\sqcup (a_1+h_1)$. Likewise, we find an annulus $A_2$ in $V_2$ with
boundary $\partial A_2=C_2\sqcup (h_2+a_2)$.
With the chosen orientations, the $A_i$ intersect each
Hopf fibre in $\Int(V_i)\setminus C_i$ once and positively.

Under the identification of $\partial V_1$ with $\partial V_2$, the
segment $a_1$ is mapped to $-a_2$. This allows us to glue
$A_1$ and $A_2$ along these boundary segments to obtain an oriented
pair of pants $P$ with boundary $\partial P=C_1\sqcup C_2\sqcup(h_1+h_2)$.
Since $h_1+h_2$ is a negative Hopf fibre, we have again
a $1$-section with one negative and two positive boundary components.

\begin{figure}[h]
\labellist
\small\hair 2pt
\pinlabel $h_1$ [r] at 97 35
\pinlabel $h_1$ [r] at 97 179
\pinlabel $h_2$ [r] at 97 89
\pinlabel $a_1=-a_2$ [l] at 118 87
\pinlabel $a_1=-a_2$ [r] at 82 135
\endlabellist
\centering
\includegraphics[scale=1.0]{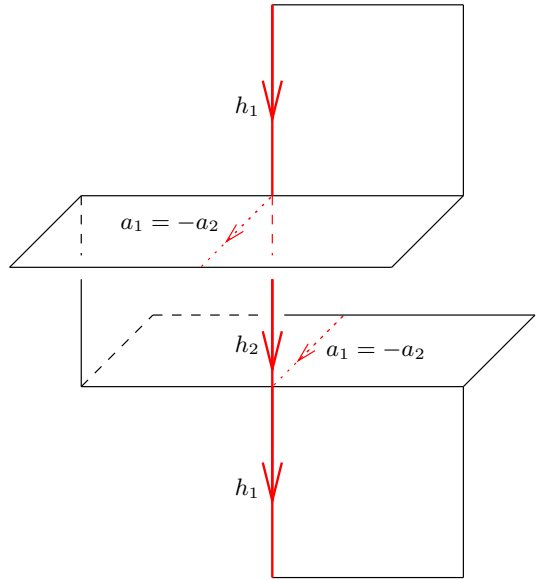}
  \caption{A piecewise linear helix.}
  \label{figure:pl-helix}
\end{figure}

Near the boundary component $h_1+h_2$ in $\partial V_1=\partial V_2$,
this surface $P$ looks as in Figure~\ref{figure:pl-helix}. This
is a piecewise smooth surface that can be smoothened rel boundary fibre
into a helicoidal surface, see Figure~\ref{figure:smoothen}.
From Figure~\ref{figure:pl-helix}
we also see that this helicoidal surface is a $1$-section,
for the Hopf fibres near $-(h_1+h_2)$ are parallel curves with respect to
the surface framing given by the $2$-torus $\partial V_1=\partial V_2$.
Notice that in Figure~\ref{figure:pl-helix} this $2$-torus (near the
fibre $-(h_1+h_2)$) is given by the vertical plane determined by
that fibre and the line segments $a_1=-a_2$; the solid torus $V_2$
sits to the left of this plane, $V_1$ sits to the right.

\begin{figure}[h]
\centering
\includegraphics[scale=.45]{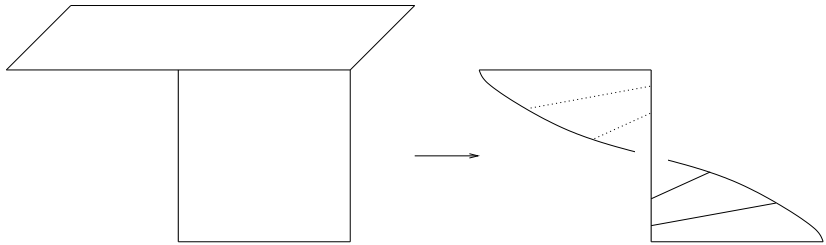}
  \caption{Smoothening rel boundary fibre.}
  \label{figure:smoothen}
\end{figure}

\section{Lens spaces}
\label{section:lens}
The lens space $L(p,q)$, for $p\in\N$ and
$q$ an integer coprime with $p$, is the oriented three-manifold defined as the
quotient of $S^3$ (with its natural orientation as
boundary of the $4$-ball in~$\C^2$) under the $\Z_p$-action generated by
\[ (z_1,z_2)\longmapsto (\rme^{2\pi\rmi/p}z_1,\rme^{2\pi\rmi q/p}z_2).\]
Since this action commutes with the Hopf flow, the flow descends
to the quotient. However, the $S^1$-action on the quotient defined
by the Hopf flow will not, in general, be free, so it only defines
a Seifert fibration on $L(p,q)$. For a classification of the
Seifert fibrations on lens spaces see~\cite{gela18}.
\subsection{The lens spaces $L(d,1)$}
The $\Z_d$-action on $S^3$ that yields $L(d,1)$ as the quotient
is the one generated by $\Psi_{2\pi/d}$, where $\Psi_t$ is the
Hopf flow from~(\ref{eqn:hopf-flow}). Here the $\Z_d$-action is along the
Hopf fibres, so the Hopf fibration descends
to the quotient to give $L(d,1)$ the structure of an $S^1$-bundle over
$S^2$ of Euler number~$-d$.
This is consistent with the description of $L(d,1)$ as the manifold
obtained from $S^3$ by surgery along an unknot with framing~$-d$,
see~\cite[p.~158]{gost99}.

This $S^1$-fibration on $L(d,1)$ corresponds to writing it
as the Seifert manifold $L(d,1)=M\bigl(0;(1,d)\bigr)$.
Indeed, the gluing $\mu_1=-\mu_2+h$, which gave us $S^3$
in Section~\ref{subsection:hopf-seifert}, becomes $\mu_1=-\mu_2+dh'$
with respect to the shortened fibre~$h'$.
\subsection{The classification of $1$-sections}
\label{subsection:class1}
The $S^1$-fibration of the lens space $L(d,1)$, including
$S^3=L(1,1)$, coming from the Hopf fibration
can be written as
\begin{equation}
\label{eqn:hopfd}
M\bigl(0;\underbrace{(1,1),\ldots,(1,1)}_{d+k},
\underbrace{(1,-1),\ldots (1,-1)}_k\bigl)
\end{equation}
with any $k\in\N_0$. This description gives rise to a
$1$-section with $k$ negative and $d+k$ positive boundaries.
Indeed, let $\Sigma_0$ be the $2$-sphere $S^2$ with $d+2k$
open discs removed. Write the boundary of $\Sigma_0$
with the opposite of its natural orientation as
\[ -\partial\Sigma_0= S^1_0\sqcup\ldots\sqcup S^1_{d+2k}.\]
The Seifert bundle (\ref{eqn:hopfd}) is then obtained by
gluing $d+2k$ solid tori $V_i=S^1\times D^2$ to the trivial
$S^1$-bundle $S^1\times\Sigma_0$ by gluing fibres to fibres
(which in the $V_i$ are given by the $S^1$-factor),
and the meridian $\mu_i$ of $V_i$ to $S^1_i\pm h$,
where $h$ denotes the fibre class, and the sign is positive for
$i=1,\ldots,d+k$, negative for $i=d+k+1,\ldots, d+2k$.
This means that $S_i=\mu_i\mp h$.

Write $C_i=S^1\times\{0\}$ for the soul of the solid torus $V_i$.
In $V_i$ we have a helicoidal surface with boundary
$\pm C_i\sqcup (\mu_i\mp h)$. These helicoidal surfaces glue with
$\Sigma_0$ to form a $1$-section for the Hopf flow on~$L(d,1)$. 

\begin{prop}
\label{prop:1section}
Any $1$-section for the Hopf flow on $L(d,1)$ is isotopic to one
of those genus~$0$ surfaces just described.
\end{prop}

\begin{proof}
Let $\Sigma$ be a $1$-section with $k_+$ positive and $k_-$
negative ends. Remove solid tori $V_i$ around the boundary fibres.
In $V_i$, the $1$-section $\Sigma$ has to look like
a helicoidal surface with boundary $\pm C_i\sqcup (\mu_i\mp h)$;
this is a consequence of one boundary of $\Sigma\cap V_i$ being $\pm C_i$,
and the fact that $\Sigma$ is a $1$-section
(cf.\ Figure~\ref{figure:helicoid}).

The part $\Sigma_0$ of $\Sigma$ lying outside the interiors of the
$V_i$ defines a trivialisation of the $S^1$-bundle there, so we can
write $L(d,1)\setminus\cup \Int(V_i)$ as $S^1\times\Sigma_0$. The
identification of the boundary components of $\Sigma_0$ (with
orientation reversed) with the $\mu_i\mp h$ completely determines
the gluing of $S^1\times\Sigma_0$ with the~$V_i$. Each such gluing
contributes $\pm 1$ to the Euler number, so we must have
$k_+=d+k_-$. It follows that $\Sigma$ is, up to diffeomorphism,
one of the surfaces we described above.

Given two such $1$-sections with $d+k$ positive and $k$ negative
boundaries, we can first isotope them so as to make
the boundaries coincide, since any finite set of distinct points on $S^2$
can be isotoped to any other set of the same cardinality.
Near a positive (resp.\ negative) boundary, the $1$-sections
look like left-handed (resp.\ right-handed) helicoidal surfaces making
one full turn; any two such surfaces are isotopic.

As before, use one of the two $1$-sections to trivialise the complement of
open solid tori around the boundary components. In this trivialised
complement $S^1\times\Sigma_0$, the boundary of the
other $1$-section $\Sigma_0'$ coincides with that of $\Sigma_0$, which
implies that $\Sigma_0$ and $\Sigma_0'$ are isotopic rel
boundary.
\end{proof}
\subsection{$d$-Sections in $S^3$ descend to $L(d,1)$}
The following statement will allow us to analyse $d$-sections for
the Hopf flow on $S^3$ via their induced $1$-sections in $L(d,1)$.

\begin{prop}
\label{prop:d-invariant}
Any $d$-section for the Hopf flow $\Psi_t$ on $S^3$ is isotopic to one
that is invariant under the $\Z_d$-action generated by
$\Psi_{2\pi/d}$ and hence descends to a $1$-section in $L(d,1)$.
\end{prop}

\begin{proof}
Near its boundary circles, a $d$-section looks like a helicoidal surface
making $d$ full turns about the central fibre given by the boundary
curve. Any such surface is isotopic to a $\Z_d$-invariant helicoid.
The remaining part of the $d$-section is a $d$-fold covering of
a punctured sphere $\Sigma_0$, embedded transversely to
the fibres in $S^1\times\Sigma_0$. By isotoping (rel boundary)
along the fibres we can ensure that any two adjacent intersections
along a fibre occur at a distance $2\pi/d$.
\end{proof}

\begin{cor}
Any positive $d$-section for the Hopf flow on $S^3$ is a surface
with $d$ boundary components.
\end{cor}

\begin{proof}
By Proposition~\ref{prop:d-invariant}, any positive $d$-section descends
to a positive $1$-section in $L(d,1)$. The latter has $d$ boundary
components by the classification of $1$-sections in
Proposition~\ref{prop:1section}.
\end{proof}
\subsection{Examples of invariant $d$-sections}
\label{subsection:ex-inv-d}
Before we classify the $d$-sections for the Hopf fibration,
we look at two examples.
\subsubsection{A positive $2$-section}
The annular $2$-section described in Section~\ref{subsubsection:annular2}
is composed of two helicoidal pieces about the boundary fibres
$C_1$ and $C_2$, glued along their other boundary curves with
the identification  $h-2\mu_1=2\mu_2-h$. As we pass to the
$\Z_2$-quotient, the two solid tori $V_1$, $V_2$ become solid tori
with fibre $h'$ of half the length of~$h$. The gluing curves descend
to $h'-\mu_1=\mu_2-h'$, or $\mu_1=-\mu_2+2h'$. This, as explained
in Section~\ref{subsection:hopf-seifert}, corresponds to the
Seifert fibration $M\bigl(0;(1,2)\bigr)$, which is the $S^1$-bundle
over $S^2$ of Euler class $-2$, i.e.\ $L(2,1)$.
\subsubsection{A positive $3$-section}
\label{subsubsection:pos3}
We now want to exhibit a positive $3$-section of genus~$1$ with three
boundary components.

\vspace{1mm}

(i) We first use a description as in
Section~\ref{subsection:pp1}~(iii), see Figure~\ref{figure:three-section}.
We think of $S^3$ as being obtained by gluing two solid tori $V_1$, $V_2$
with the identification $\mu_1=\lambda_2$, $\lambda_1=\mu_2$.
The Hopf fibration is given by the two souls $C_1,C_2$ and
the $(1,1)$-curves on the Hopf tori.

\begin{figure}[h]
\labellist
\small\hair 2pt
\pinlabel $\lambda_1$ [r] at 0 105
\pinlabel $\mu_1$ [t] at 105 0
\pinlabel $\lambda_2$ [r] at 323 105
\pinlabel $\mu_2$ [t] at 428 0
\pinlabel $h_1^1$ [br] at 28 28
\pinlabel $h_1^2$ [br] at 100 100
\pinlabel $h_1^3$ [br] at 172 172
\pinlabel $a_1$ [bl] at 64 14
\pinlabel $a_1$ [tr] at 118 176
\pinlabel $c_1$ [bl] at 182 41
\pinlabel $b_1$ [tr] at 49 104
\pinlabel $b_1$ [tr] at 210 158
\pinlabel $h_2^1$ [br] at 389 66
\pinlabel $h_2^3$ [tl] at 518 200
\pinlabel $h_2^2$ [br] at 461 138
\pinlabel $a_2$ [tr] at 353 49
\pinlabel $a_2$ [bl] at 515 102
\pinlabel $b_2$ [bl] at 443 31
\pinlabel $b_2$ [tr] at 491 201
\pinlabel $c_2$ [bl] at 397 149
\endlabellist
\centering
\includegraphics[scale=.5]{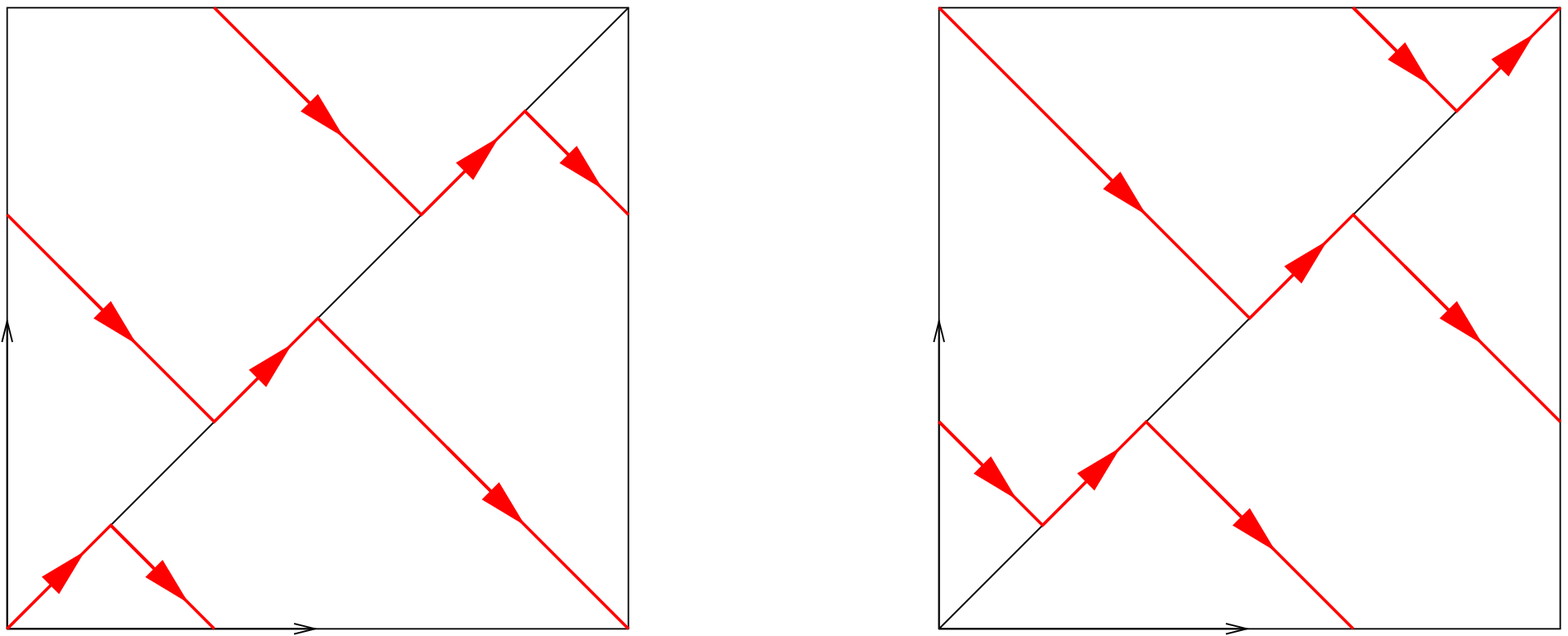}
  \caption{A positive $3$-section of genus~$1$.}
  \label{figure:three-section}
\end{figure}

Write $\sigma_1$ for the curve on $\partial V_1$ made up of
the straight line segments $h_1^1$, $a_1$, $h_1^3$, $b_1$, $h_1^2$, $c_1$.
Similarly, the curve $\sigma_2$ on $\partial V_2$ is made up of
$a_2$, $h_2^1$, $b_2$, $h_2^3$, $c_2$, $h_2^2$.
Notice that $\sigma_i$ is a $(2,-1)$-curve on $\partial V_i$ with respect to
the basis $(\mu_i,\lambda_i)$. 
In $V_1$ we have a helicoidal annulus $A_1$ with boundary
$\partial A_1=C_1\sqcup\sigma_1$; in $V_2$, an annulus $A_2$
with $\partial A_2=C_2\sqcup \sigma_2$.

Under the identification of $\partial V_1$ with $\partial V_2$,
the segments $a_1$, $b_1$, $c_1$ are mapped to $-a_2$, $-b_2$, $-c_2$.
This allows us to glue $A_1$ and $A_2$ along these segments to
obtain a surface $\Sigma$ with boundary consisting of
three positive fibres: $C_1$, $C_2$ and the one made up
of the segments $h_i^j$. Near this third fibre,
$\Sigma$ can be smoothened as in
Figure~\ref{figure:smoothen}.

The surface $\Sigma$ is a positive $3$-section. Indeed, $\Sigma$
intersects the Hopf tori in $(2,-1)$-curves; the intersection number of
these curves with the Hopf fibres, which are $(1,1)$-curves, is
\[ (2,-1)\bullet (1,1)=3.\]
From $2\mu_i-\lambda_i=3\mu_i-h$ we see that near $C_i$ the surface
$\Sigma$ looks like a left-handed helicoid making three
full turns along the fibre, as it should. The same is true for
the third component of $\partial\Sigma$,
as can bee seen from the explicit gluing construction in
Figure~\ref{figure:three-section} and a comparison with
Figure~\ref{figure:pl-helix}.

The surface $\Sigma$ is invariant under the $\Z_3$-action generated
by $\Psi_{2\pi/3}$, and hence descends to a $1$-section in $L(3,1)$
with three positive boundaries, i.e.\ a pair of pants.

This $3$-section $\Sigma$ is topologically a surface of genus~$1$.
There are many ways to see this. One is to observe that $\Sigma$ is obtained
by gluing two annuli along three segments in one boundary component of
each annulus. This is the same as joining the annuli by one-handles.
Joining the two annuli with a single one-handle is the same as
attaching two one-handles to a two-disc so as to create a pair of
pants. We then attach two further
one-handles to the two-disc such that the `outer' boundary stays connected
(since this is the boundary of the helicoidal surface about the
fibre made up of the~$h_i^j$) and the surface is
orientable. This is a $2$-torus with three
discs removed, see Figure~\ref{figure:three-section-topology}
or the discussion in~\cite{geig17}.

\begin{figure}[h]
\labellist
\small\hair 2pt
\endlabellist
\centering
\includegraphics[scale=.5]{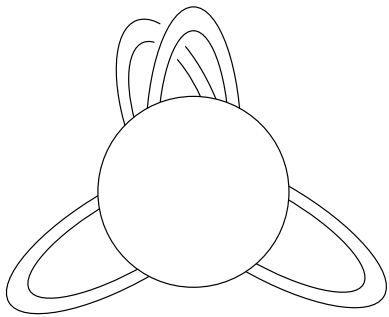}
  \caption{The topology of the $3$-section.}
  \label{figure:three-section-topology}
\end{figure}

Alternatively, we can appeal to the Riemann--Hurwitz formula.
We formulate the relevant result in full generality
for positive $d$-sections.

\begin{prop}
\label{prop:Sigmagd}
Let $\Sigma_{g,d}$ be the connected, closed, orientable surface of genus
$g$ with $d$ open discs removed, and $S^2_d$ the $2$-sphere with $d$
open discs removed. There is a $d$-fold unbranched covering
$\Sigma_{g,d}\rightarrow S^2_d$ if and only if $g=(d-1)(d-2)/2$.
\end{prop}

\begin{proof}
The `if' part will follow from the construction of a positive $d$-section
below. For the `only if' part, we extend the unbranched covering
$\Sigma_{g,d}\rightarrow S^2_d$ to a branched covering
$\Sigma_g\rightarrow S^2$ with $d$ branch points upstairs, each of
branching index~$d$. Then, by the Riemann--Hurwitz formula, the Euler
characteristic of $\Sigma_g$ is
\[ 2-2g=\chi(\Sigma_g)=d\bigl(\chi(S^2)-d\bigr)+d=3d-d^2,\]
and hence $g=(d-1)(d-2)/2$.
\end{proof}

\begin{rem}
In Section~\ref{section:pos-d} we give a proof not only of the
`if' part of Proposition~\ref{prop:Sigmagd}, but also of the
`only if' part, directly from the classification of
positive $d$-sections, which does not use the Riemann--Hurwitz formula.
\end{rem}

(ii) Here is an alternative construction of the positive $3$-section
$\Sigma$ as a lift of the positive $1$-section $\oSigma$ in $L(3,1)$.
This construction has the advantage of generalising to all~$d$,
while in (i) we made essential use of the fact that there
was only one boundary fibre apart from $C_1,C_2$, which we could place
on the Hopf torus $\partial V_1=\partial V_2$.

Recall from Proposition~\ref{prop:1section} that $\oSigma$
is a surface of genus~$0$ with three boundary components, i.e.\ a pair
of pants. Near any of these boundary components, $\oSigma$
looks like a left-handed helicoid making one full turn
along the fibre, see Figure~\ref{figure:one-section}.

\begin{figure}[h]
\labellist
\small\hair 2pt
\endlabellist
\centering
\includegraphics[scale=.45]{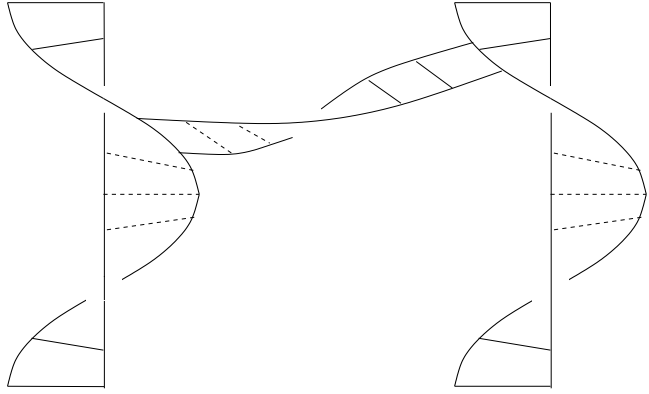}
  \caption{The $1$-section $\oSigma$ in $L(3,1)$ near two boundary fibres.}
  \label{figure:one-section}
\end{figure}

Consider two of these three helicoids. They project to discs in~$S^2$.
Join these discs by a band as shown in Figure~\ref{figure:sigma0}.
Over this part of $S^2$ the bundle $L(3,1)\rightarrow S^2$
is trivial, and the two helicoids can be joined by
a band to form a $1$-section.

\begin{figure}[h]
\labellist
\small\hair 2pt
\endlabellist
\centering
\includegraphics[scale=.4]{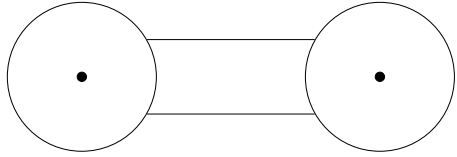}
  \caption{The projection of $\oSigma$ to~$S^2$.}
  \label{figure:sigma0}
\end{figure}

The lift of this part of $\oSigma$ to $S^3$ is shown in
Figure~\ref{figure:three-covering}. Observe that this surface
in $S^3$ has three boundary components: the two special fibres
and one further connected component.

\begin{figure}[h]
\labellist
\small\hair 2pt
\endlabellist
\centering
\includegraphics[scale=.6]{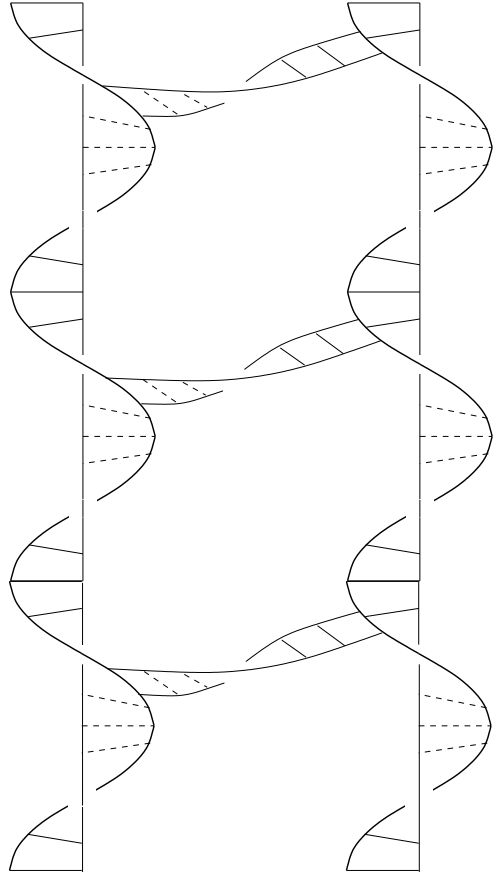}
  \caption{Alternative view of the positive $3$-section.}
  \label{figure:three-covering}
\end{figure}

If we write $\mu_0$ for the boundary of the disc in $S^2$ shown
in Figure~$7$ (oriented positively, i.e.\ counter-clockwise),
and $h$ for the Hopf fibre in~$S^3$, then this third boundary component
represents the class $3\mu_0-2h$ in the $2$-torus in $S^3$
sitting over~$\mu_0$.

Now consider a small disc around the base point in $S^2$ of
the third boundary fibre of $\oSigma$. We denote the boundary
of this disc by~$\mu_3$. Over this disc, $\oSigma$ forms
a left-handed helicoid making one full turn; its lift $\Sigma$
to $S^3$ makes three full turns. Thus, the boundary of this
helicoid on the $2$-torus in $S^3$ sitting over $\mu_3$
is the curve $3\mu_3-h$.

In order to obtain the $3$-section $\Sigma$ in $S^3$ we need to
glue the part shown in Figure~\ref{figure:three-covering}
with this third helicoid by identifying $3\mu_0-2h$ with
$-(3\mu_3-h)$. This amounts to the same as gluing $\mu_0$
with $-\mu_3+h$, which --- as discussed in
Section~\ref{subsection:hopf-seifert} ---
corresponds to the description of $S^3$ as $M\bigl(0;(1,1)\bigr)$,
so we have indeed found a positive $3$-section in~$S^3$.
\section{Positive $d$-sections}
\label{section:pos-d}
We now want to extend the examples from Section~\ref{subsection:ex-inv-d}
to all natural numbers~$d$, and then give a classification of these
$d$-sections.
\subsection{Construction of a positive $d$-section}
In order to obtain a positive $d$-section for the Hopf flow,
we need to replace the two $3$-helicoids in
Figure~\ref{figure:three-covering} by $d-1$ left-handed helicoids
making $d$ full twists, joined in sequence by $d$ bands between
any two successive helicoids.

First of all, we want to observe that the boundary of this
oriented surface consists of one connected component besides the
$d-1$ boundary fibres where the helicoids are attached. Start at a
boundary point at the top right of Figure~\ref{figure:three-covering}
(generalised to~$d$). Each time we walk along the boundary of a
horizontal band and continue along the boundary of the next helicoid
to the left, we move down one level on these helicoids. After having reached
the left-most helicoid, we move back to the right along horizontal
bands, staying on the same level. When we have returned to the right-most
helicoid, we move down one more level until we arrive again
at a band going to the left. In the base, this path projects to one
full passage along the outer boundary of~$\oSigma$.

Thus, with each such turn, we have moved down $d-1$ levels. After
completing $d$ full turns in the base,
the lifted path has covered the whole boundary upstairs. So this
boundary upstairs is connected, and it represents the class
$d\mu_0-(d-1)h$ on the boundary of the solid torus sitting
over~$\oSigma$. As at the end of the preceding section,
we see that this accords with the description of the
Hopf fibration as $M\bigl(0;(1,1)\bigr)$.

After joining the $d-1$ helicoids in sequence by a single band between any
two successive helicoids, we have a surface consisting of $d-1$ one-handles
attached to a single two-disc. We then add a further $(d-1)(d-2)$ one-handles,
ending up with an oriented surface with $d$ boundary components. Thus,
the genus of this surface is
\[ g=\frac{1}{2}\bigl((d-1)^2-(d-1)\bigr)=\frac{1}{2}(d-1)(d-2).\]
This proves the `if' part of Proposition~\ref{prop:Sigmagd}.
\subsection{The classification of positive $d$-sections}
By considering the choices in the above construction, we arrive
at the following classification result. In particular, this reproves the
`only if' part of Proposition~\ref{prop:Sigmagd}.

\begin{thm}
\label{thm:d-section}
For each $d\in\N$ there is, up to isotopy, a unique positive $d$-section
for the Hopf flow on~$S^3$. It is a connected, orientable surface
of genus $(d-1)(d-2)/2$ with $d$ boundary components.
\end{thm}

\begin{proof}
Let $\Sigma$ be a positive $d$-section. Given two distinct points
$x,y\in\Sigma$, consider their images $\ox,\oy\in S^2$ under
the Hopf projection $S^3\rightarrow S^2$. Join either of $\ox$ and $\oy$
by a path in $S^2$ to the base point $\oz$ of a boundary fibre of~$\Sigma$
(such a boundary exists, as observed at the beginning of
Section~\ref{section:hopf-flow}). These paths lift to paths in $\Sigma$
joining both $x$ and $y$ with the component of $\partial\Sigma$
over~$\oz$. This proves that $\Sigma$ is connected.

An example of a positive $d$-section with the claimed topological
properties was exhibited above, so it only remains to prove
uniqueness up to isotopy. 
Given two positive $d$-sections, by Proposition~\ref{prop:1section}
we may assume, after an isotopy, that they project to the same $1$-section
in~$L(d,1)$. In particular, the $d$ boundary fibres of the two surfaces
sit over the same $d$ points in $S^2$. Therefore it suffices to
show that there are no choices, up to isotopy, in the construction
of a positive $d$-section we described.

The $d$ lifted $d$-helicoids near the boundary fibres are
determined by $1$-helicoids of the $1$-section in $L(d,1)$.
The $d-2$ bands connecting $d-1$ of these helicoids in
$L(d,1)$ into a chain lift uniquely to $d$ times $d-2$ bands
in~$S^3$ as in Figure~\ref{figure:three-covering}:
start with a helicoid at the end of the chain and look at the
$d$ lifted bands to the neighbouring helicoid. Shifting this
helicoid along the fibre by a suitable multiple of $2\pi/d$ will
make the bands `horizontal', so we obtain the standard picture
as shown in Figure~\ref{figure:three-covering}.

In the remaining construction we join this partial $d$-section
by an annulus with the helicoid around the last boundary fibre.
Again, there are no choices up to isotopy.
\end{proof}

\section{Complex projective curves}
\label{section:curves}
In this section we study how algebraic curves $C\subset\CP^2$ give rise
to surfaces of section $\Sigma\subset S^3$ for the Hopf flow. The surface
$\Sigma$ is obtained, under suitable assumptions on~$C$,
by radially projecting the affine part $\Ca:=C\cap\C^2$
of the algebraic curve, with the origin $(0,0)\in\C^2$ removed if it happens
to lie on~$\Ca$, onto the unit sphere $S^3\subset\C^2$.

Intersection points of $C$ with the complex projective line at
infinity will correspond to positive boundary components of~$\Sigma$.
If $\Ca$ avoids the origin in~$\C^2$, there will be no negative
boundaries; if $(0,0)\in\Ca$, this will give rise to negative boundaries.
\subsection{Projecting to $S^3$}
\label{subsection:projecting}
For $(a,b)\in S^3\subset\C^2$,
the Hopf circle $\Psi_t(a,b)$, $t\in\R/2\pi\Z$, in $S^3$
is the image of the punctured radial complex plane
\[ P_{a,b}:=\bigl\{(az,bz)\co z\in\C^*\bigr\}\]
under the radial projection $\C^2\setminus\{(0,0)\}\rightarrow S^3$.
Thus, in order to show that the projection of an algebraic curve 
$\Ca\subset\C^2\setminus\{(0,0)\}$ to $S^3$ intersects a Hopf circle
in $d$ distinct points, we need to show that $\Ca$ intersects
the corresponding plane $P_{a,b}$ in $d$ distinct points,
no two of which lie on the same real ray
\[ \bigl\{r\rme^{\rmi t}(a,b)\co r\in\R^+\bigr\}.\]

Any complex plane $P$ through the origin in $\C^2$ determines a
point $P_{\infty}$ in the complex projective line $\CP_{\infty}^1$ at
infinity and vice versa. In order to show that the projection of $\Ca$
to $S^3$ is a positive $d$-section for the Hopf flow, we need to verify
that the projected surface in $S^3$ becomes asymptotic to the
Hopf circles $P\cap S^3$ corresponding to the intersection
points $P_{\infty}\in C\cap\CP_{\infty}^1$ (and there should be no
intersection points of $P$ with $\Ca$ in this case). The positivity
of the $d$-section is ensured by the positivity of complex
intersections.

In order to understand the asymptotic behaviour of (not necessarily
positive) surfaces of
section near their boundary, it will be useful not to look at the
projection of $\Ca\setminus\{(0,0)\}$ to~$S^3$, but rather to regard
$\Ca\setminus\{(0,0)\}$ as a surface in $\R\times S^3$ under
the identification of $\C^2\setminus\{(0,0)\}$ with the
symplectisation $\bigl(\R\times S^3,\rmd(\rme^{2s}\alphast)\bigr)$
of~$(S^3,\alphast)$, with $\alphast$ as in~(\ref{eqn:alphast}).
This identification is given by sending
the flow lines of the radial vector field
\[ X=\frac{1}{2}(x_1\partial_{x_1}+y_1\partial_{y_1}+
x_2\partial_{x_2}+y_2\partial_{y_2})\]
on $\C^2$ to those of~$\partial_s/2$.
The vector field $X$ is a Liouville vector field for the standard symplectic
form $\omegast=\rmd x_1\wedge\rmd y_1+\rmd x_2\wedge\rmd y_2$
on~$\C^2$, that is, $L_X\omegast=\omegast$, and it is homothetic for
the standard metric. Therefore, the described identification
of $\C^2\setminus\{(0,0)\}$ with $\R\times S^3$ sends the complex structure
on $\C^2$ to the standard almost complex structure $J$ on the
symplectisation. This $J$ preserves the contact structure $\ker\alphast$
and, with $\Rst$ denoting the Reeb vector field of~$\alphast$,
it satisfies $J\partial_s=\Rst$, since $\rmi X=\Rst/2$ along $S^3\subset\C^2$.

\subsection{Homogeneous affine polynomials}
We begin with the simple situation that the affine curve $\Ca=C\cap\C^2$
is described by a homogeneous polynomial of degree~$d$.
We write $[z_0:z_1:z_2]$ for the homogeneous coordinates on $\CP^2$.

\begin{prop}
\label{prop:complex-d}
The complex projective curve $C=\{F=0\}\subset\CP^2$, where $F$ is a
complex polynomial of the form
\[ F(z_0,z_1,z_2)=f(z_1,z_2)-z_0^d,\]
with $f\neq 0$ a homogeneous polynomial of degree~$d$,
defines a positive $d$-section for the Hopf flow on $S^3$ if and only
if $F$ is non-singular.
\end{prop}

\begin{proof}
We first determine the intersection points of the affine part
\[ \Ca=\biggl\{(z_1,z_2)\in\C^2\co f(z_1,z_2)=1\biggr\} \]
(which does not contain the origin $(0,0)\in\C^2$)
with the radial planes $P_{a,b}$ in $\C^2$.
These intersection points are given by the equation
\[ f(a,b)z^d=1.\]
If $f(a,b)\neq 0$, this equation has $d$ solutions~$z$, related by
multiplication by a power
of the $d^{\mathrm{th}}$ root of unity. Otherwise, there are no
solutions.

The partial derivatives of $F$ are given by
\[ \frac{\partial F}{\partial z_0}=-dz_0^{d-1},\;\;\;
\frac{\partial F}{\partial z_1}=\frac{\partial f}{\partial z_1},\;\;\;
\text{and}\;\;\;
\frac{\partial F}{\partial z_2}=\frac{\partial f}{\partial z_2}.\]
For $z_0\neq 0$ we have $\partial F/\partial z_0\neq 0$,
so the affine part is always non-singular.

We now look at the points of $C$ in
\[ \CP_{\infty}^1=\bigl\{[z_0:z_1:z_2]\in\CP^2\co z_0=0\bigr\}.\]
Solutions of $F=0$ of the form $[0:a:b]$ are determined by the
equation $f(a,b)=0$. In other words, a point at infinity
lies on $C$ precisely when $C$ does \emph{not} intersect the
radial plane in $\C^2$ determined by that point.

Now, the projection of $\Ca$ to $S^3$ extends to a positive $d$-section
precisely when it is asymptotic to $d$ distinct Hopf orbits.
This amounts to saying that the equation $f(a,b)=0$ should have
$d$ distinct solutions $[a:b]\in\CP^1$, which is equivalent to
$f$ being non-singular. This, in turn, is equivalent to $F$
being non-singular.

It remains to check that $\Ca$ has the correct asymptotic behaviour
near these $d$ Hopf orbits. Let $[a_1:b_1]\in\CP^1$ be a solution
of $f(a,b)=0$. We may assume without loss of generality that $a_1\neq 0$.
For a small $\varepsilon>0$ the curve
\[ \theta\longmapsto[a_{\theta}:b_{\theta}]:=
[a_1:b_1+\varepsilon a_1\rme^{\rmi\theta}],\;\;\;\theta\in S^1=\R/2\pi\Z,\]
describes a circle in $\CP^1$ around the point $[a_1:b_1]$.

The points of $\Ca$ (projected to~$S^3$) in the Hopf fibre
over $[a_{\theta}:b_{\theta}]$ are given by
the solutions $w_{\theta}$ of the equation
$f(a_{\theta},b_{\theta})w_{\theta}^d=1$, and then radially projecting the
points $(a_{\theta}w_{\theta},b_{\theta}w_{\theta})$ to~$S^3$.
As $\theta$ makes one full turn in~$S^1$, the function
$\arg\bigl(f(a_{\theta},b_{\theta})\bigr)$ likewise makes one
complete turn, provided $\varepsilon>0$ is
sufficiently small. This can be seen by factorising $f$ as
\[ f(z_1,z_2)=(b_1z_1-a_1z_2)\cdots(b_dz_1-a_dz_2)\]
with the $a_j,b_j$ describing
$d$ distinct points $[a_j:b_j]\in\CP^1$.

Thus, if we choose a solution $w_0$ and then define
$w_{\theta}$, $\theta\in\R$, continuously in $\theta$, we have
$w_{\theta+2\pi}=\rme^{-2\pi\rmi/d}w_{\theta}$. This guarantees
that the projection of $\Ca$ to $S^3$ does indeed look like
a left-handed $d$-fold helicoid about a Hopf fibre near each of its $d$
boundary components.
\end{proof}

In particular, for $d=1$ the polynomial $F$ describes
a projective line $L\neq\CP^1_{\infty}$, since $f\neq 0$.
This line has a single point at infinity, and the projection of
the affine part $L_{\mathrm{a}}=L\cap\C^2$ to $S^3$ defines
a disc-like $1$-section for the Hopf flow.

\begin{rem}
\label{rem:asymptote}
For the final part of the proof of Proposition~\ref{prop:complex-d},
the asymptotic behaviour near the boundary components, one may
also look at the behaviour of $\Ca$ near $s=\infty$ under
the identification of $\C^2\setminus\{(0,0)\}$ with $\R\times S^3$
described in Section~\ref{subsection:projecting}. The tangent spaces
of $\Ca$ contain vectors getting closer and closer to $\partial_s$
as we approach $s=\infty$, and hence also tangent vectors close
to the Reeb vector field $\Rst=J\partial_s$. This suffices to see that
$\Ca$ becomes asymptotic to a Reeb orbit, but it does not guarantee,
as our \emph{ad hoc} argument does, that this orbit will only be simply
covered.
\end{rem}
\subsection{Algebraic curves giving rise to $1$-sections}
We next want to describe a class of homogeneous polynomials
$F(z_0,z_1,z_2)$ of degree $d$ that give rise to $1$-sections
for the Hopf flow with $d$ positive and $d-1$ negative boundary
components.

Write $f_k(z_1,z_2)$ for a non-zero homogeneous polynomial
of degree~$k$. As in the previous section, we can factorise
this as
\[ f_k(z_1,z_2)=c_k(b_1^kz_1-a_1^kz_2)\cdots(b_k^kz_1-a_k^kz_2)\]
with $c_k\in\R^+$ and $(a_j^k,b_j^k)\neq (0,0)$. The factor $c_k$
in this expression allows us to assume without loss of
generality that the $(a_j^k,b_j^k)$ lie in $S^3\subset\C^2$.

The following is easy to see.

\begin{lem}
\label{lem:cylinder}
Let $C=\{F=0\}\subset\CP^2$ be the algebraic curve defined by
\[ F(z_0,z_1,z_2)=f_d(z_1,z_2).\]
With notation as above, we assume
that the $[a_j^d:b_j^d]\in\CP^1$ are pairwise distinct
for $j=1,\ldots, d$. Then $\Ca\setminus\{(0,0)\}\subset\R\times S^3$
defines a collection of $d$ cylinders $\R\times\gamma$ over the Hopf fibres
$\gamma$ through the points $(a_j^d,b_j^d)\in S^3$.
\qed
\end{lem}

Next we look at polynomials defined by a pair $f_d,f_{d-1}$.

\begin{prop}
\label{prop:complex-d-1}
Let $F$ be a homogeneous complex polynomial of degree $d$ of the form
\[ F(z_0,z_1,z_2)=f_d(z_1,z_2)+z_0f_{d-1}(z_1,z_2),\]
and $C=\{F=0\}\subset\CP^2$.
With notation as above, we assume that the $[a_j^k:b_j^k]\in\CP^1$
are pairwise distinct for $k\in\{d-1,d\}$ and $1\leq j\leq k$.
Then the projection of $\Ca\setminus\{(0,0)\}$ to $S^3$ defines
a $1$-section for the Hopf flow with $d$ positive and $d-1$
negative boundary components, given by the Hopf fibres through
the points $(a_j^d,b_j^d)$ and $(a_j^{d-1},b_j^{d-1})$,
respectively.
\end{prop}

\begin{proof}
Observe that $C$ is non-singular, since
a common zero $[z_0:z_1:z_2]$ of $F$ and $\partial F/\partial z_0=f_{d-1}$
would also have to be a zero of~$f_d$, which our assumptions rule out.

The case $d=1$ is covered by Proposition~\ref{prop:complex-d}, so we
assume $d\geq 2$ from now on.

There are $d$ distinct points
at infinity on the curve~$C$, as $C\cap\CP^1_{\infty}$ is given
by the equation $f_d=0$. The intersection of $\Ca$ with
a (punctured) radial plane $P_{a,b}$ is described by the equation
\[ f_d(a,b)z+f_{d-1}(a,b)=0,\;\; z\neq 0.\]
There are no solutions if $f_d(a,b)=0$, since this would force
a common zero with $f_{d-1}$. Likewise, there is no solution
if $f_{d-1}(a,b)=0$. For $f_d(a,b),f_{d-1}(a,b)\neq 0$, there
is a unique intersection point of $\Ca$ with~$P_{a,b}$.
This proves that $\Ca\setminus\{(0,0)\}$ projects to
a $1$-section for the Hopf flow away from the Hopf fibres
over the points $[a_j^k:b_j^k]$.

For the asymptotic behaviour near these fibres, we
consider a small circle $\theta\mapsto
[a_{\theta}:b_{\theta}]\in\CP^1$ around a solution $[a:b]$
of $f_d=0$ or $f_{d-1}=0$, as in the proof of
Proposition~\ref{prop:complex-d}. The point of the $1$-section in the Hopf
fibre over $[a_{\theta}:b_{\theta}]$ is given by radially projecting the
point $(a_{\theta}w_{\theta},b_{\theta}w_{\theta})$ to~$S^3$,
with $w_{\theta}$ determined by
\[ w_{\theta}=
-\frac{f_{d-1}(a_{\theta},b_{\theta})}{f_d(a_{\theta},b_{\theta})}.\]
As we encircle a zero of $f_d$, the argument of $w_{\theta}$ makes
one negative rotation; around a zero of~$f_{d-1}$, a positive one.
Thus, near these fibres the $1$-section looks like a left-handed
resp.\ right-handed helicoid.
\end{proof}

\begin{rem}
\label{rem:complex-d-1}
(1) By Proposition~\ref{prop:1section}, the $1$-sections found in
Proposition~\ref{prop:complex-d-1} are of genus~$0$.

(2) For the asymptotic behaviour of $\Ca\setminus\{(0,0)\}\subset
R\times S^3$ near $s=-\infty$ we may alternatively observe that
as $\Ca\ni(z_1,z_2)\rightarrow (0,0)$, the surface becomes asymptotic
to the surface given by $f_{d-1}=0$, which by Lemma~\ref{lem:cylinder}
is a cylinder over a Hopf fibre.
The same caveat as in Remark~\ref{rem:asymptote} applies.
\end{rem}

\section{Holomorphic buildings and the degree-genus formula}
\label{section:buildings}
In this section we present two proofs of the degree-genus formula.

\begin{thm}
\label{thm:dg}
Any non-singular algebraic curve $C\subset\CP^2$ of degree $d$
is homeomorphic to a closed, connected orientable surface of genus
$g=(d-1)(d-2)/2$.
\end{thm}

One proof only uses the classification of $d$-sections for the
Hopf flow. The second proof uses degenerations of complex algebraic
curves into holomorphic buildings in the sense of symplectic field
theory. This second proof yields an explanation of the degree-genus formula
as a sum $\sum_{k=1}^{d-2}k$. Either proof relies on the fact that,
as a consequence of Bertini's theorem~\cite[Lecture~17]{harr92},
the general (in the sense of~\cite[p.~53]{harr92})
algebraic curve of degree~$d$ in $\CP^2$ is non-singular.

The projective space of homogeneous polynomials $F(z_0,z_1,z_2)$ of
degree $d$ is of dimension $N=(d^2+3d)/2$, since there
are $(d+2)(d+1)/2$ monomials of degree $d$ in three variables.
There is an embedding $\CP^2\rightarrow\CP^N$ given by sending
the point $[z_0:z_1:z_2]$ to $[...:z^I:...]$, where $z^I$ ranges over
all monomials of degree~$d$ in three variables. The image of this
embedding is the Veronese variety~\cite[p.~23]{harr92}, which
is a smooth variety.

The algebraic curves of degree $d$ in $\CP^2$ are exactly the hyperplane
sections of the Veronese variety. To this description of
algebraic curves one can apply Bertini's theorem on the
smoothness of hyperplane sections to conclude that the subset
of non-singular algebraic curves in the space of all algebraic curves
of degree $d$ is open, dense, and connected. Under deformations
through non-singular curves, the topological genus is invariant.

A slightly more direct (and more sophisticated) argument can be based
on the version of Bertini's theorem proved in
\cite[Corollary~III.10.9]{hart77}. The projective space of
degree $d$ homogeneous polynomials in three variables (or the set
of divisors made up of the curves defined by these polynomials)
is a linear system (see also \cite[Section~1.1]{grha78} for
a discussion of linear systems more accessible to non-algebraic
geometers). This linear system is without base points, i.e.\
for every point in $\CP^2$ there is an algebraic curve of degree~$d$
\emph{not} containing the given point. Then Bertini's theorem says that
almost every element of this linear system, that is,
every element outside a lower-dimensional subvariety, is non-singular.

\begin{proof}[First proof of Theorem~\ref{thm:dg}]
The algebraic curves $C$ of degree $d$ in Proposition~\ref{prop:complex-d}
have $d$ distinct points at infinity, and their affine part $\Ca$
does not contain the origin. The projection of $\Ca$ to
$S^3$ defines a positive $d$-section. By Theorem~\ref{thm:d-section}
this means that, when viewed in $\R\times S^3$, the complex curve $\Ca$ is
topologically a connected, orientable surface of genus
$g=(d-1)(d-2)/2$ with $d$ ends asymptotic to cylinders over Hopf fibres.
The algebraic curve $C$ is obtained topologically by capping off these
ends with discs.

This proves the degree-genus formula for the algebraic curves
described in Proposition~\ref{prop:complex-d}. For the general case, it
suffices to appeal to the connectedness of the space of non-singular curves
of degree~$d$.
\end{proof}

Our second proofs illustrates the degeneration phenomena in
symplectic field theory.

\begin{proof}[Second proof of Theorem~\ref{thm:dg}]
After a projective transformation of $\CP^2$ we may assume that
$[1:0:0]\not\in C$. Then $C$ can be written as $\{F=0\}$ with $F$
of the form
\[ F(z_0,z_1,z_2)=f_d(z_1,z_2)+z_0f_{d-1}(z_1,z_2)+\cdots+
z_0^{d-1}f_1(z_1,z_2)+z_0^d.\]

By a small perturbation of $F$ we may assume that each $f_k$
has $k$ distinct zeros, and no adjacent pair
$f_k,f_{k-1}$ has zeros in common. In particular,
the intersection $C\cap\CP^1_{\infty}$ then consists of $d$
non-singular points, and we shall focus our attention on the
affine part~$\Ca$. As before, topologically the closed surface
$C$ is obtained by capping off the $d$ ends of $\Ca$ with discs.
Since the subspace of singular curves is of real
codimension~$2$ by Bertini's theorem, we may further assume that
the whole family
\begin{eqnarray*}
f^{\lambda}(z_1,z_2)
 & := & f_d(z_1,z_2)+\lambda f_{d-1}(z_1,z_2)+
          \lambda^3 f_{d-2}(z_1,z_2)+\cdots\\
 &    & \cdots+\lambda^{d(d-1)/2}f_1(z_1,z_2)
          +\lambda^{(d+1)d/2},\;\;\;\lambda\in(0,1],
\end{eqnarray*}
where the power of $\lambda$ multiplying $f_{d-k}$ is
$\sum_{j=0}^k j$,
consists of non-singular polynomials. Notice that none
of these curves $C^{\lambda}=\{f^{\lambda}=0\}$
contains the origin in~$\C^2$, so we may think
of this as a family of curves $C^\lambda\subset\R\times S^3$.

Our aim is to determine the topological genus of the affine curve
$\{f^1=0\}$, which is a curve with $d$ boundary components.
In the naive limit $\lambda\rightarrow 0$ we lose all topological
information, since by Lemma~\ref{lem:cylinder} the curve
$\{f_d=0\}\setminus\{(0,0)\}$ is simply a collection
of $d$ cylinders, for as $\lambda\rightarrow 0$, the topology of
$C^{\lambda}$ disappears towards $-\infty$ in $\R\times S^3$.

In the spirit of SFT~\cite{behwz03}, we now rescale the curve in different
ways during this limit process $\lambda\rightarrow 0$, which amounts to
zooming in at different parts of the curve to see its topology.
We first present the heuristic argument; details of the
convergence process will be discussed in Section~\ref{section:convergence}.

For the rescaling, we replace $(z_1,z_2)$ by $c_{\lambda}(z_1,z_2)$,
with judicious choices of the scaling factor~$c_{\lambda}$.
The rescaling leads to the family of polynomials
\[f^{\lambda}_*=c_{\lambda}^d f_d+\lambda c_{\lambda}^{d-1} f_{d-1}+
          \lambda^3 c_{\lambda}^{d-2} f_{d-2}+\cdots
+\lambda^{d(d-1)/2} c_{\lambda}f_1
          +\lambda^{(d+1)d/2}.\]

We now choose $c_{\lambda}=\lambda^k$ for some
$1\leq k\leq d$. Then the polynomials $f_{d-k+1}$ and
$f_{d-k}$ are multiplied by the same power
\[ k(d-k+1)+\sum_{j=0}^{k-1} j=k(d-k)+\sum_{j=0}^k j\]
of $\lambda$,
whereas all other summands contain a larger power of $\lambda$.
Hence, as $\lambda\rightarrow 0$ the rescaled polynomial
\[ f^\lambda_*/\lambda^{k(d-k)+\sum_{j=0}^k j}\]
converges to $f_{d-k+1}+f_{d-k}$, which for $k=d$ has to be read
as $f_1+1$. By Remark~\ref{rem:complex-d-1}~(1),
this defines a surface of genus $0$ with $d-k+1$ positive
and $d-k$ negative boundaries at $\pm\infty$, respectively,
in $\R\times S^3$.

As shown in Proposition~\ref{prop:complex-d-1},
the curve $\{f_{d-k+1}+f_{d-k}=0\}\setminus\{(0,0)\}$
is asymptotic to the Hopf fibres determined by the zeros
of $f_{d-k+1}$ and $f_{d-k}$ at $+\infty$ and $-\infty$, respectively.
Hence, these limits for the different choices of rescaling
$c_{\lambda}$ fit together into a holomorphic building in the sense
of SFT as shown in Figure~\ref{figure:building}.

\begin{figure}[h]
\labellist
\small\hair 2pt
\pinlabel $f_4+f_3=0$ [tl] at 468 450
\pinlabel $f_3+f_2=0$ [tl] at 468 286
\pinlabel $f_2+f_1=0$ [tl] at 468 125
\pinlabel $f_1+1=0$ [tl] at 468 22
\endlabellist
\centering
\includegraphics[scale=.5]{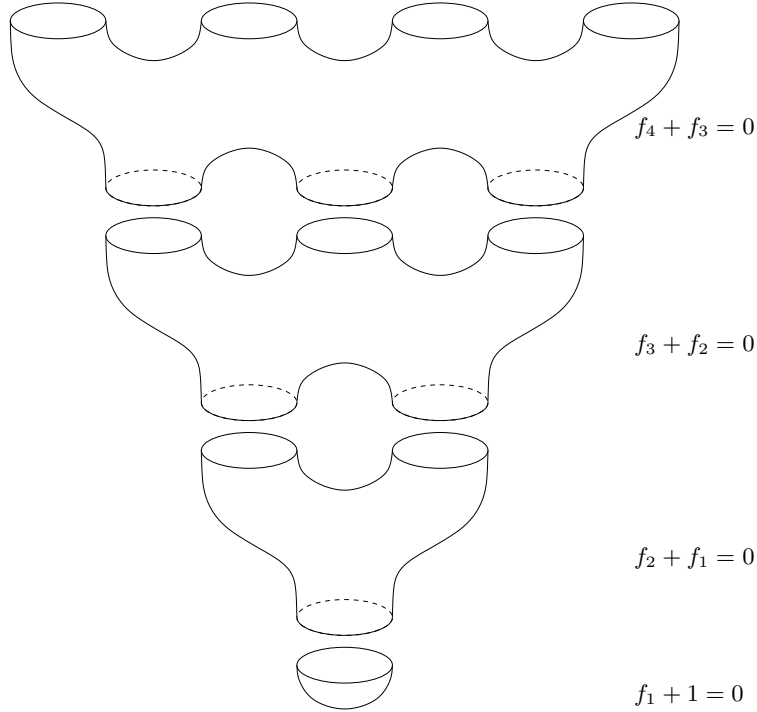}
  \caption{A holomorphic building from a degree $4$ curve.}
  \label{figure:building}
\end{figure}

Observe that intermediate rescalings only lead to trivial cylinders
over the boundary orbits and hence do not carry any additional topology.
For instance, if we choose $c_{\lambda}=\lambda^{3/2}$, then
\[ f^{\lambda}_*=\lambda^{3d/2}f_d+\lambda^{(3d-1)/2}f_{d-1}+
\lambda^{3d/2}f_{d-2}+\cdots,\]
and after rescaling only the polynomial $f_{d-1}$ will
survive in the limit. Then refer to Lemma~\ref{lem:cylinder}.

Thus, the genus of $C$ can be read off the holomorphic building
we obtain in this SFT limit. The individual levels carry no genus,
and the gluing of two adjacent levels adds $\#(\text{limit orbits})-1$
to the genus. We conclude that the genus of the curve $C$
of degree $d$ is given by $\sum_{k=1}^{d-2} k=(d-1)(d-2)/2$.
\end{proof}

\begin{ex}
Here is a concrete example that illustrates the essential aspects
in the following discussion of convergence. Suppose we would like
to understand the topology of the Fermat curve of degree~$3$,
\[ \bigl\{[z_0:z_1:z_2]\in\CP^2\co z_0^3+z_1^3+z_2^3=0\bigr\}.\]
We consider the affine part
\[ \bigl\{(z_1,z_2)\in\C^2\co z_1^3+z_2^3+1=0\bigr\}.\]
We now introduce terms of lower order and a family parameter~$\lambda$:
\[ f^{\lambda}(z_1,z_2)=z_1^3+z_2^3+\lambda(z_1^2+z_2^2)+
\lambda^3(z_1+z_2)+\lambda^6.\]
When we evaluate $f^{\lambda}$ at $(\lambda z_1,\lambda z_2)$, we
obtain
\[ f^{\lambda}(\lambda z_1,\lambda z_2)=
\lambda^3(z_1^3+z_2^3+z_1^2+z_2^2)+\lambda^4(z_1+z_2)+\lambda^6;\]
rescaling with $\lambda^2$ yields
\[ f^{\lambda}(\lambda^2 z_1,\lambda^2 z_2)=
\lambda^6(z_1^3+z_2^3)+\lambda^5(z_1^2+z_2^2+z_1+z_2)+\lambda^6;\]
the third rescaling to consider is
\[ f^{\lambda}(\lambda^3 z_1,\lambda^3 z_2)=
\lambda^9(z_1^3+z_2^3)+\lambda^7(z_1^2+z_2^2)+\lambda^6(z_1+z_2+1).\]
After dividing these polynomials by $\lambda^3,\lambda^5$ and $\lambda^6$,
respectively, we see that in the limit $\lambda\rightarrow 0$ we
obtain the respective polynomials
\[ z_1^3+z_2^3+z_1^2+z_2^2,\;\;\;
z_1^2+z_2^2+z_1+z_2,\;\;\;
z_1+z_2+1.\]
\end{ex}
\section{SFT convergence}
\label{section:convergence}
In this section we fill in the technical details of the
second proof of Theorem~\ref{thm:dg}.
\subsection{Convergence of submanifolds}
\label{subsection:conv-submanifold}
In order to understand the convergence of submanifolds defined
by equations, we consider the following general situation.
Let $M\subset\R^n$ be a compact submanifold of codimension $k$
defined globally by $k$ smooth functions $h_1,\ldots, h_k\co
\R^n\rightarrow\R$. This means that
\[ M=\{h_1=\ldots=h_k=0\},\]
with the gradient vector fields $\nabla h_1,\ldots,\nabla h_k$
pointwise linearly independent along the common zero set~$M$
of the~$h_i$. In particular, the normal bundle of $M$ is trivial,
and we find a tubular neighbourhood $\nu M$ of $M\subset\R^n$ diffeomorphic
to $M\times D^k$ such that at each point of $M\times D^k$
the orthogonal complement to the span of $\nabla h_1,\ldots,\nabla h_k$
is transverse to the $D^k$-factor.

We may assume that there is an $\varepsilon >0$ such that
at any point outside the tubular neighbourhood $\nu M$, at least one of
the functions $|h_i|$ takes a value larger than~$\varepsilon$.
Now let $d_1,\ldots,d_k\co\R^n\rightarrow [-1,1]$ be smooth functions.
Then, for $|\lambda|<\varepsilon$,
the common zero set of the functions $h_i+\lambda d_i$ lies
inside $\nu M$. By shrinking $\nu M$ and $\varepsilon$
we can ensure that the gradient vector fields
$\nabla h_i+\lambda\nabla d_i$ are pointwise linearly
independent on $\nu M$ for any $|\lambda|<\varepsilon$,
and the orthogonal complement to their span is transverse to the
$D^k$-factor.

Under these assumptions, the common zero set
\[ M_{\lambda}=\{h_1+\lambda d_1=\ldots=h_k+\lambda d_k=0\} \]
will be a submanifold contained in $\nu M=M\times D^k$
for $|\lambda|<\varepsilon$, given as the graph of a map
$M\rightarrow D^k$. In particular, $M_\lambda$ will be
an isotopic copy of~$M$.
\subsection{Degeneration of algebraic curves}
We now return to the specific situation of Section~\ref{section:buildings}.
We write $(z_1,z_2)=\rme^t q$ with $t\in\R$ and $q\in S^3\subset\C^2$.
Set
\[ g_k=f_k|_{S^3},\;\; k=0,\ldots, d,\]
where $f_0=1$. Then
\[ G^{\lambda}(t,q):=f^{\lambda}(\rme^t q)=
\sum_{\ell=0}^d \rme^{(d-\ell)t}\lambda^{\ell(\ell+1)/2} g_{d-\ell}(q).\]

The rescaling of $(z_1,z_2)$ by a constant factor amounts to a shift in
the $t$-coordi\-nate, so we set
\[ G^{\lambda}_{\mu}(t,q)=g^{\lambda}(t+\mu\log\lambda,q)=
\sum_{\ell=0}^d\rme^{(d-\ell)t}\lambda^{\mu(d-\ell)+\ell(\ell+1)/2}
g_{d-\ell}(q).\]
The choice $\mu=k$ corresponds to $f^{\lambda}_*$ with $c_{\lambda}=\lambda^k$
in the second proof of Theorem~\ref{thm:dg}.
\subsection{Convergence to a holomorphic building}
With this choice $\mu=k$ we want to get a quantitative understanding of
the convergence of the
rescaled function
\begin{equation}
\label{eqn:Gk}
\frac{G^{\lambda}_k(t,q)}{\lambda^{k(d-k)+k(k+1)/2}}=
\frac{1}{\lambda^{k(d-k)+k(k+1)/2}}
\sum_{\ell=0}^d \rme^{(d-\ell)t}\lambda^{k(d-\ell)+\ell(\ell+1)/2}
g_{d-\ell}(q)
\end{equation}
to
\begin{equation}
\label{eqn:survivor}
G^0_k(t,q)=\rme^{(d-k+1)t}g_{d-k+1}(q)+\rme^{(d-k)t}g_{d-k}(q)
\end{equation}
for $\lambda\rightarrow 0$. Notice that the summands in
(\ref{eqn:Gk}) that vanish in the limit
are of the form
\[ \lambda^m\rme^{(d-k+1+n)t}g_{d-k+1+n}(q)\;\;\;\text{or}\;\;\;
\lambda^m\rme^{(d-k-n)t}g_{d-k-n}(q)\]
with $m\geq n>0$. On any compact interval $t\in[-N,N]$, these
summands go uniformly to zero for $\lambda\rightarrow 0$, but we can do a
little better than that.

For large positive $t$, the first summand in (\ref{eqn:survivor})
dominates, so we consider the rescaled function
\begin{equation}
\label{eqn:G+}
G^+_k(t,q)=g_{d-k+1}(q)+\rme^{-t}g_{d-k}(q);
\end{equation}
for $t<0$ with $|t|$ large, we look at
\[ G^-_k(t,q)=\rme^tg_{d-k+1}(q)+g_{d-k}(q).\]

\begin{lem}
On $[0,-\frac{3}{4}\log\lambda]\times S^3$, the rescaled function
\[ \frac{G^{\lambda}_k(t,q)}{\lambda^{k(d-k)+k(k+1)/2}\,\rme^{(d-k+1)t}}\]
converges uniformly to $G^+_k(t,q)$ for $\lambda\rightarrow 0$.

On $[\frac{3}{4}\log\lambda,0]\times S^3$, the rescaled function
\[ \frac{G^{\lambda}_k(t,q)}{\lambda^{k(d-k)+k(k+1)/2}\,\rme^{(d-k)t}}\]
converges uniformly to $G^-_k(t,q)$ for $\lambda\rightarrow 0$.
\end{lem}

\begin{proof}
For $t\in [0,-\frac{3}{4}\log\lambda]$ and $m\geq n>0$ as above,
we have
\[ \lambda^m \rme^{nt}<\lambda^{m-3n/4}\longrightarrow 0\]
for $\lambda\rightarrow0$. The other case is analogous.
\end{proof}

\begin{rem}
Notice that the domain of convergence increases as $\lambda$ gets smaller.
By uniform convergence we mean that for any $\varepsilon >0$
there is a $\lambda_0=\lambda_0(\varepsilon)$ such that for any
$\lambda<\lambda_0$ the function (\ref{eqn:G+}) is $\varepsilon$-close
to $G_k^+(t,q)$ for all $(t,q)\in[0,-\frac{3}{4}\log\lambda]\times S^3$,
similarly for the other case.
This statement remains true for any finite number of derivatives,
with a smaller $\lambda_0(\varepsilon)$.
\end{rem}

The considerations of Section~\ref{subsection:conv-submanifold}
now imply that for $\lambda$ sufficiently close to~$0$,
the curve
\[ C^{\lambda}\cap[k+\frac{3}{4}\log\lambda,k-\frac{3}{4}\log\lambda] \]
has the topology of $\bigl\{f_{d-k+1}+f_{d-k}=0\bigr\}$. The intervals
$[k+\frac{3}{4}\log\lambda,k-\frac{3}{4}\log\lambda]$ overlap for
adjacent~$k$, and similar considerations show that in the region of
overlap the topology of $C^{\lambda}$ is that of a collection of cylinders
over Reeb orbits.

This concludes the convergence argument in the second proof
of Theorem~\ref{thm:dg}.
\begin{ack}
This paper was initiated during an enjoyable stay at Schlo{\ss}
Rauischholzhausen, the conference centre of JLU
Gie{\ss}en. We thank the castle staff for creating an inspiring research
environment. We are grateful to Stefan Kebekus for useful correspondence
on algebraic curves, and to Jean Gutt for comments
on a draft version of this paper.
\end{ack}
\end{document}